\documentclass[11pt,amssymb]{amsart}
\usepackage{amsmath,amssymb,amsfonts,euscript,enumerate}
\usepackage{amsthm}
\usepackage{color,wrapfig}
\usepackage{pxfonts}
\usepackage{verbatim}
\usepackage{ulem}
\usepackage{graphicx}
\usepackage{color} 
\newcommand{\nc}{\newcommand}

\newtheorem{thm}{Theorem}[section]

\newtheorem{corollary}[thm]{Corollary}

\newtheorem{lemma}[thm]{Lemma}

\newcommand{\beq}{\begin{equation}}
\newcommand{\eeq}{\end{equation}}
\newcommand{\bcl}{\begin{center}}
\newcommand{\ecl}{\end{center}}
\newcommand{\re}{\mathbb{R}}

\newcommand{\La}{\triangle}
\newcommand{\D}{\nabla}

\newcommand{\R}{{\mathbb R}}
\newcommand{\vp}{\varphi}
\newcommand{\oa}{\overline{\alpha}}
\newcommand{\al}{\alpha}
\newcommand{\be}{\beta}
\newcommand{\bs}{\backslash}
\nc \sms {\smallskip}
\nc{\fp}{\noindent}
\nc{\qd}{\qquad\qquad}
\title{Geometric Properties of  Gelfand's  Problems with Parabolic Approach}

\author[Sunghoon Kim]{Sunghoon Kim}
\address{Sunghoon Kim:
Department of Mathematics, School of Natural Sciences, The Catholic University of Korea,
43 Jibong-ro, Wonmi-gu, Bucheon-si(city), Gyeonggi-do, 420-743, Republic of Korea}
\email{math.s.kim@catholic.ac.kr}

\author[Ki-ahm Lee]{Ki-ahm Lee}
\address{Ki-ahm Lee:
School of Mathematical Sciences,
Seoul National University,
San56-1 Shinrim-dong Kwanak-gu Seoul 151-747,South Korea
}
\email{kiahm@math.snu.ac.kr}

\begin{document}

\maketitle
\begin{abstract}

\noindent  We consider the asymptotic profiles of the nonlinear
parabolic flows 
$$(e^{u})_{t}= \La u+\lambda e^u$$ 
to show the geometric properties of minimal solutions of the following elliptic nonlinear eigenvalue problems known as a Gelfand's problem:
\begin{equation*}
\begin{split}
\La \vp &+ \lambda e^{\vp}=0, \quad \vp>0\quad\text{in $\Omega$}\\
\vp&=0\quad\text{on $\Omega$}
\end{split}
 \end{equation*}
posed in a strictly convex domain $\Omega\subset\re^n$. In this work, we show that there is a strictly increasing function $f(s)$ such that $f^{-1}(\vp(x))$ is convex for $0<\lambda\leq\lambda^{\ast}$, i.e., we prove that level set of $\vp$ is convex. Moreover, we also present the boundary condition of $\vp$ which guarantee the $f$-convexity of solution $\vp$.
\end{abstract}


\setcounter{equation}{0}
\setcounter{section}{0}

\section{Introduction}\label{sec-intro}
\setcounter{equation}{0}
\setcounter{thm}{0}
We will investigate the geometric properties of parabolic flows and derive related  geometric properties for the asymptotic limits of such evolutions. More precisely, we consider the nonnegative solutions $u(x,t)$ of the following equation
\begin{equation}\label{eq-main-for-u-3}
\left(e^u\right)_t=\La u+\lambda e^u
\end{equation}
posed on a strictly convex and bounded domain $\Omega$ with zero boundary condition
\begin{equation}\label{eq-zero-boundary-condition-134533}
u=0 \qquad \qquad \mbox{on $\partial\Omega$}, 
\end{equation}
and initial data 
\begin{equation}\label{eq-initial-condition-134533}
u(x,0)=u_0(x)>0
\end{equation}
In the limit, these flows converge to solutions of the well-known Gelfand's problems 
\begin{equation}\label{eq-main-1}
\begin{cases}
\La \vp(x) =-\lambda e^{ \vp(x)}\quad\text{in $\Omega$}\\
\vp(x)=0\quad\text{on $\partial\Omega$.}
\end{cases}
\tag{{\bf $GP_{\lambda}$}}
\end{equation}
By this relation, it is natural to expect that the solution of the parabolic flow above have a lot in common with those of the Gelfand's problems, \eqref{eq-main-1}. The aim of the paper is to provide geometric properties of solutions of the Gelfand's problems by using the parabolic method which is introduced by Lee and V\'azquez, \cite{LV}.\\
\indent  Recently, there has been a lot of studies of the problem \eqref{eq-main-1} because of its wide applications. It arises in many physical models: it describes problems of thermal self-ignition \cite{Ge}, a ball of isothermal gas in gravitational equilibrium proposed by lord Kelvin \cite{Ch}, the problem of temperature distribution in an object heated by the application of a uniform electric current \cite{KC} and Osanger's vortex model for turbulent Euler flows \cite{CLMP}.

Let us summarize some known results to \eqref{eq-main-1}. It is well-known  that there exists a finite positive number $\lambda^{\ast}$, called the {\it extremal value}, such that Gelfand's problem \eqref{eq-main-1} has at least a classical solution which is minimal among all possible positive solutions if $0\leq\lambda<\lambda^{\ast}$, while no solution exists, even in the weak sense, for $\lambda>\lambda^{\ast}$. Let us call the minimal solution $\underline{\vp}_{\lambda}$. The family of such solutions depends smoothly and monotonically on $\lambda$, and in particular
\begin{equation*}
\underline{\vp}_{\lambda}<\underline{\vp}_{\lambda'} \qquad \mbox{if $\lambda<\lambda'$}.
\end{equation*}
Their limit as $\lambda\nearrow\lambda^{\ast}$ is the external solution $\vp_{\lambda^{\ast}}$ and it can be either classical or singular (i.e. unbounded). It is known that $\vp^{\ast}\in L^{\infty}(\Omega)$ for every $\Omega$ if $n\leq 9$, while $\vp^{\ast}=\log\left(\frac{1}{|x|^2}\right)$ and $\lambda^{\ast}=2(n-2)$ if $n\geq 10$ and $\Omega=B_1$. Brezis and V\'azquez \cite{BV} investigated the existence and regularities of extremal solutions when they are unbounded. The regularity theory for the Gelfand's problem was improved by S. Nedev \cite{Ne} who  proved that, for general smooth domain $\Omega$, $\vp^{\ast}$ is a classical solution if $n\leq 3$, while $\vp^{\ast}\in H^{1}_0(\Omega)$ if $n\leq 5$. \\
\indent The ultimate goal in this article is to establish the geometric properties of $\underline{\vp}_{\lambda}$ for $\lambda<\lambda^{\ast}$. Especially, we'd like to show that
\begin{equation*}
f(\underline{\vp}_{\lambda})=e^{-\frac{1}{2}\underline{\vp}_{\lambda}}\,\, : \quad \mbox{convex}.
\end{equation*}     
From now on, we call it $f$-convexity shortly. We also refer to the minimal solution $\underline{\vp}_{\lambda}$ of the Gelfand's problem \eqref{eq-main-1} as $\vp$ for the rest of the paper. Since the minimal solution $\vp$ can be obtained as the limit of solution $u$ to \eqref{eq-main-for-u-3}-\eqref{eq-initial-condition-134533} as $t\to\infty$, we will concentrate on showing $f$-convexity of $u$ under the assumption that the initial value $u_0$ has the following property
\begin{equation*}
f(u_0)=e^{-\frac{1}{2}u_0}\,\, : \quad \mbox{strictly convex}.
\end{equation*}
\indent Large number of literatures on the convexity properties of the solutions of semilinear elliptic equations can be found. We refer the reader to the papers of Caffarelli and Friedmann \cite{CF} and  Korevaar \cite{Ko1}, \cite{Ko2} for the geometric results which are related with our work.\\ 
\indent The parabolic approximation method introduced in \cite{LV} relies on the fact that the nonlinear elliptic problem \eqref{eq-main-1} can be described the asymptotic profile of a corresponding parabolic flow in a bounded domain and we use that possibility as follows: we select an initial data for the parabolic flow having the desired geometric property. Then, the corresponding solution $u$ of \eqref{eq-main-for-u-3} will converge eventually to the minimal solution $\vp$ as $t\to\infty$. If the evolution preserves the $f$-convexity property under investigation, the result for the problem \eqref{eq-main-1} will be obtained in the limit $t\to\infty$.\\
\indent To investigate the $f$-convexity of solutions to the corresponding parabolic flow, we will split it into two steps. The step 1 will be devoted to the study of the $f$-convexity of solution $u$ on the boundary. On $\partial\Omega$, the second derivatives of $f(u)$ can be written in the form
\begin{equation*}
\left[f(u)\right]_{\alpha\alpha}=\frac{1}{2}e^{-\frac{1}{2}u}\left(\frac{1}{2}u_{\alpha}^2-u_{\alpha\alpha}\right), \qquad \left(D_{e_{\alpha}}u=u_{\alpha}\right).
\end{equation*}
Hence, the geometric properties of solutions on the boundary can be determined by the balances between quantities, $u_{\alpha}$ and $u_{\alpha\alpha}$. However, the difference between them is very subtle in this problem. Therefore, there is little room for perturbing quantities. Thus, it is very difficult to investigate the geometric properties of solution on the boundary without some boundary condition. In this paper, we will focus on the solutions of \eqref{eq-main-for-u-3} having the conditions not only \eqref{eq-zero-boundary-condition-134533}, \eqref{eq-initial-condition-134533} but also
\begin{equation}\label{boundary-condition-of-gelfand-with-PME}
G(u,\lambda,\Omega)=\frac{1}{2}u_{\nu}^2+\lambda+(n-1)u_{\nu}H(\partial\Omega)+Ku_{\nu}>0 \qquad \mbox{on $\partial\Omega$} 
\end{equation}
for sufficiently large $K>0$ where $\nu$ is the outer normal vector to $\partial\Omega$ and $H(\partial\Omega)$ is the mean curvature of $\Omega$ at the boundary. Here the constant $K$ is related to the shape of boundary $\partial\Omega$.\\
\indent As the second step, we extend the geometric result on the boundary to the interior of $\Omega$. Since the continuity of the second derivatives of $f(u)$, it is expected that if there are some problems or difficulties then they may occur at the region far away from the boundary. Hence, the equation that describe the parabolic flow plays a very important mathematical role in the study of the geometric properties of solution in the interior.
\subsection{Outlines}
This paper is divided into three parts: In Part 1 (Section 2) we study the convexity of solution to the degenerate equation
\begin{equation}\label{eq-degenerated-equation-in-outlines}
w_t=\rho(w)a^{ij}(\D w)w_{ij}+b_1(w)|\D w|^p+b_2(w)|\D w|^2+b_3(w) 
\end{equation}
on a strictly convex domain $\Omega$. It is a simple observation that the convexity of solution will be strongly effected by the coefficients. Hence, proper conditions need to be imposed to the coefficients of \eqref{eq-degenerated-equation-in-outlines} for the result. The Part 2 (Section 3) is devoted to the proof for the $f$-convexity of the minimal solution to the Gelfand's problem with boundary condition \eqref{boundary-condition-of-gelfand-with-PME}. As will be mentioned later, the main equation of Gelfand's problem is a special form of \eqref{eq-degenerated-equation-in-outlines}. Thus, the proof will be focused on the $f$-convexity of solution on the boundary. In Part 3, we will discuss the boundary condition \eqref{boundary-condition-of-gelfand-with-PME} of a solution 

\section{Convexity for degenerate equation}\label{Section-Convexity-for-degenerate-equation}

\setcounter{equation}{0}
\setcounter{thm}{0}

In this section, we will study degenerated equations of the form
\begin{equation}\label{eq-w}
w_t=\rho(w)a^{ij}(\D w)w_{ij}+b_1(w)|\D w|^p+b_2(w)|\D w|^2+b_3(w), \qquad \left(p\geq 2\right)
\end{equation}
on the bounded cylinder $\Omega\times[0,\infty)$, where $\Omega$ is a bounded strictly convex domain in $\R^n$ with smooth boundary. 
The subindices $i,j\in\{1,\cdots,n\}$ denote differentiation with respect to the space variables $x_1$, $\cdots$, $x_n$ and the summation convention is used. We assume that the coefficient matrix $\left(a^{ij}\right)$ is strictly positive and all coefficients $a^{ij}$, $b$, $c$ and $d$ belong to appropriate $C^k$, $(k=1,\cdots)$ space which will be defined later. The degeneracy of the equation is carried through the function $\rho(w)$ which is assumed to be smooth on $\Omega$. \\
We assume further that the coefficients of \eqref{eq-w} satisfies the following conditions:
\begin{equation*}
\begin{aligned}
&{\bf I.1} \qquad \qquad \qquad \qquad b_1(s),\,\, b_2(s)\,\, \mbox{and}\,\, b_3(s) \quad \mbox{are convex,} \qquad \qquad \qquad \qquad\\
&{\bf I.2} \qquad \qquad \qquad \qquad \qquad \rho''-\frac{\left(\rho'\right)^2}{2\rho}=0.\qquad \qquad \qquad \qquad
\end{aligned}
\end{equation*}
Denoting by $L$ the operator
\begin{equation*}
Lw=w_t-\left(\rho(w)a^{ij}(\D w)w_{ij}+b_1(w)|\D w|^p+b_2(w)|\D w|^2+b_3(w)\right), \qquad \left(p\geq 2\right), 
\end{equation*}
we can now state the main result in this section:
\begin{lemma}\label{lemma-convexity-for-the-general-case-0}
Let $\Omega$ be a strictly convex bounded domain in $\R^n$ with smooth boundary and  suppose that the coefficients $a^{ij}$, $b$, $c$ and $d$ of the operator $L$ are smooth and satisfy the elliptic condition
\begin{equation*}
a^{ij}\xi_i\xi_j\geq c_0\left|\xi\right|^2>0 \qquad \forall\xi\in\R^n\bs\{0\}
\end{equation*}
for some positive constant $c_0$. In addition, assume that $\rho$ is a smooth function in $\R$ and strictly positive on $\R\bs\{0\}$. Let $w$ be a positive smooth solution of \eqref{eq-w} satisfying the conditions conditions {\bf I.1-2}. If $w$ is convex on the parabolic boundary of $\Omega\times[0,\infty)$, i.e., if $\inf_{\partial_l \Omega\times[0,\infty)}\inf_{|e_{\al}|=1}w_{\al\al}(x,t)>0$ and $\inf_{\Omega}\inf_{|e_{\al}|=1}w_{\al\al}(x,0)>0$, then $w$ is convex in the space variable for all $t>0$, i.e., $\inf_{\Omega\times[0,\infty)}\inf_{|e_{\al}|=1}w_{\al\al}(x,t)\geq 0$.
\end{lemma}
\begin{proof}
By direct computation, \eqref{eq-w} implies that the evolution of $w_{\alpha\beta}$ is given by the equation
\begin{equation}
\begin{split}
w_{,\alpha\beta t}=&\rho(w)a^{ij}(\D w)w_{\alpha\beta,ij}+\rho(w)a^{ij}_k(\D w)(w_{\alpha k} w_{\beta i j}+w_{\beta k} w_{\alpha ij})\\
                &+\rho'(w)a^{ij}(\D w)(w_{\alpha} w_{\beta i j}+w_{\beta} w_{\alpha ij})+\rho'(w)w_{\alpha}a^{ij}_kw_{k\beta}D_{ij}w+\rho'(w)w_{\beta}a^{ij}_kw_{k\alpha}w_{ij}\\
                &+\rho(w)a^{ij}_k(\D w)w_{k\al\be}w_{ij}+\rho(w)a^{ij}_{kl}(\D w)w_{k\al}w_{l\be}D_{ij} w\\
                &+\rho'(w)w_{\al\be}a^{ij}(\D w)D_{ij} w+\rho''(w)w_{\alpha}w_{\beta}a^{ij}(\D w)D_{ij} w+b_1'(w)|\nabla w|^pw_{\alpha\beta}\\
                &+ \left[pb_1(w)|\D w|^{p-2}+2b_2(w)\right]\left(\nabla w_{\alpha}\cdot\nabla w_{\beta}+\nabla w\cdot \nabla w_{\alpha\beta}\right)\\
                &+p(p-2)b_1(w)w_kw_lw_{k\al}w_{l\be}\left|\nabla w\right|^{p-4}+\left[pb_1'(w)|\D w|^{p-2}+2b_2'(w)\right](w_{\alpha}\D w\cdot\D w_{\beta}+w_{\beta}\D w\cdot\D w_{\alpha})\\
                &+b_1''(w)|\D w|^{p-2}w_{\alpha}w_{\beta}+b_2'(w)w_{\alpha\beta}|\D w|^2+b_2''(w)w_{\alpha}w_{\beta}|\D w|^2+b_3'(w)w_{\alpha\beta}+b_3''(w)w_{\alpha}w_{\beta}
\end{split}
\end{equation}
To estimate the minimum of the second derivatives of $w$ with respect to space variables, we take a look at the following quantity, for a positive function $\psi(t)$,
\begin{equation*}
\inf_{x\in\Omega, s\in[0,t]}\inf_{e_{\beta}\in\R^n, |e_{\beta}|=1}\left[w_{\beta\beta}+\epsilon\psi(t)\right]=w_{\overline{\alpha}\overline{\alpha}}(x_0,t_0)+\epsilon\psi(t_0)
\end{equation*}
We need to show that there are a function $\psi(t)$ and $\epsilon_0$ such that
\begin{equation*}
w_{\overline{\alpha}\overline{\alpha}}(x_0,t_0)+\epsilon\psi(t_0)> 0 \qquad \forall 0<\epsilon<\epsilon_0.
\end{equation*}
To get a contradiction, suppose that 
\begin{equation}\label{eq-assumption-for-t-0-with-zero-level}
w_{\overline{\alpha}\overline{\alpha}}(x_0,t_0)+\epsilon\psi(t_0)=0.
\end{equation} 
Observe that the minimum is taken at $(x_0,t_0)$ with direction $\overline{\alpha}$. By the assumption that $w$ is convex on the parabolic boundary of $\Omega\times[0,\infty)$, the minimum point of $w_{\alpha\alpha}$ is little way off the boundary. Hence, we can put $x_0$ without loss of generality. On the other hand, the parabolic equation $w_{\alpha\alpha}$ contains third order derivatives which is difficult to be controlled by the information of $w_{\alpha\alpha}$. Hence we are going to perturb the direction of the derivative to create extra terms, keeping the minimum point and minimum zero. For the perturbation, we take the function which is introduced in \cite{LV}.\\
\indent We now use the function
\begin{equation*}
Z=w_{,\al\be}\eta^{\al} \eta^{\be}+\epsilon\psi(t)|\eta|^2
\end{equation*}
where the modifying functions $\eta^{\beta}(x)$ are constructed as follows: at
$x=0$ we assume that the $\eta^\beta$ satisfy the system
\begin{equation}
\eta^{\be}_{,i}=(c_{\gamma}\eta^{\gamma})\delta_{\be i}, \qquad
\eta^{\be}_{,ij}=(c_{l}\eta^{l})c_{\gamma}\delta_{\gamma
\,i}\delta_{\be j},
\label{eq-eta}
\end{equation}
where the subscripts are space derivatives. Putting also
$\eta^{\beta}(0)=\delta_{\oa\be}$, it  follows that
\begin{equation*}
\eta^{\be}_{,i}(0)=c_{\gamma}\delta_{\oa,\gamma}\delta_{\be i}, \qquad
\eta^{\be}_{,ij}(0)=c_{\oa}c_{\gamma}\delta_{\gamma i}\delta_{\be j},
\label{eq-eta2}
\end{equation*}
and
$$
\eta^{\beta}(x)=\delta_{\oa\be}+c_{\oa}x^{\be}+\frac12
c_{\oa}c_{\gamma}x^{\gamma}x^{\be}\,.
$$
Hence, at $x=0$,
\begin{equation*}
\begin{aligned}
Z_{,i}&=w_{,\al\be i}\eta^{\al}\eta^{\be}+2w_{,\be
i}c_{\al}\eta^{\al}\eta^{\be}+\epsilon\psi(t)\left(|\eta|^2\right)_i\\
Z_{,ij}&=w_{,\al\be ij}\eta^{\al}\eta^{\be}+2c_{\alpha}w_{,\be
ij}\eta^{\al}\eta^{\be}+2c_{\beta}w_{,\alpha ij}\eta^{\al}\eta^{\be}\\
&\qquad \quad +2c_jc_{\al}w_{,\be i}\eta^{\al}\eta^{\be}+2c_ic_{\al}w_{,\be j}\eta^{\al}\eta^{\be}+2c_{\al}c_{\be}w_{,ij}\eta^{\al}\eta^{\be}+\epsilon\psi(t)\left(|\eta|^2\right)_{ij}.
\end{aligned}
\end{equation*}
\begin{equation}\label{eq-w-1}
\begin{aligned}
Z_{,t}-\epsilon\psi'|\eta|^2=&\rho a^{ij}D_{ij} Z+\rho a^{ij}_k\left[w_{\alpha k} w_{\beta i j}\eta^{\al}\eta^{\be}+w_{\beta k} w_{\alpha ij}\eta^{\al}\eta^{\be}\right]\\
                &+a^{ij}\left[\rho' w_{\alpha}-2\rho c_{\alpha}\right]w_{\beta i j}\eta^{\al}\eta^{\be}+a^{ij}\left[\rho' w_{\beta}-2\rho c_{\beta}\right]w_{\alpha ij}\eta^{\al}\eta^{\be}\\
                &+a^{ij}\left[\rho'' w_{\alpha}w_{\beta}-2\rho c_{\alpha}c_{\beta}\right]w_{ij}\eta^{\al}\eta^{\be}-2\rho a^{ij}c_jc_{\al}w_{,\be i}\eta^{\al}\eta^{\be}-2\rho a^{ij}c_ic_{\al}w_{,\be j}\eta^{\al}\eta^{\be}\\
                &-\epsilon\psi a^{ij}\left(|\eta|^2\right)_{ij}+\rho a^{ij}_k w_{k\al\be}w_{ij}\eta^{\al}\eta^{\be}+\rho a_{ij}^{kl}w_{k\al}w_{l\be}w_{ij}\eta^{\al}\eta^{\be}\\
                &+\rho'w_{\al\be}a^{ij}w_{ij}\eta^{\al}\eta^{\be}+\rho' a_k^{ij}w_{\alpha}w_{k\beta}w_{ij}\eta^{\al}\eta^{\be}+\rho' a_k^{ij}w_{\beta}w_{k\alpha}w_{ij}\eta^{\al}\eta^{\be}\\
                &+ \left[pb_1(w)|\D w|^{p-2}+2b_2(w)\right](\nabla w\cdot \nabla Z)+p|\nabla w|^{p-2}\left[\left(b_1'(w)w_{\alpha}-c_{\alpha}b_1(w)\right)\D w\cdot\D w_{\beta}\eta^{\al}\eta^{\be}\right]\\
                &+p|\nabla w|^{p-2}\left[\left(b_1'(w)w_{\beta}-c_{\beta}b_1(w)\right)\D w\cdot\D w_{\alpha}\eta^{\al}\eta^{\be}\right]+\left[\left(b_2'(w)w_{\alpha}-c_{\alpha}b_2(w)\right)\D w\cdot\D w_{\beta}\eta^{\al}\eta^{\be}\right]\\
                &+\left[\left(b_2'(w)w_{\beta}-c_{\beta}b_2(w)\right)\D w\cdot\D w_{\alpha}\eta^{\al}\eta^{\be}\right]-\epsilon\psi\left[pb_1(w)|\D w|^{p-2}+2b_2(w)\right](\nabla w\cdot \nabla |\eta|^2)\\
                &\left[pb_1(w)|\D w|^{p-2}+2b_2(w)\right]\left(\nabla w_{\alpha}\cdot \nabla w_{\beta}\right)\eta^{\al}\eta^{\be}+p(p-2)b_1(w)|\nabla w|^{p-4}w_kw_lw_{k\al}w_{l\be}\eta^{\al}\eta^{\be}\\
                &+\left[b_1''(w)|\D w|^{p-2}+b_2''(w)|\D w|^2+b_3''(w)\right]w_{\alpha}w_{\beta}\eta^{\al}\eta^{\be}+\left[b_1'(w)|\nabla w|^{p}+b_2'(w)|\nabla w|^2+b_3'(w)\right]w_{\alpha\beta}\eta^{\al}\eta^{\be}.
\end{aligned}
\end{equation}
We are now going to choose $c_{\alpha}$ such that 
\begin{equation}\label{condition-of-c-alpha-with-rho-b-c}
\rho'w_{\alpha}-2c_{\alpha}\rho=0
\end{equation}
at $(x_0,t_0)$. Then, by the \textbf{conditions I} and \eqref{eq-w-1},
\begin{equation}\label{eq-w-1-after-first-calculation}
\begin{aligned}
Z_{,t}-\epsilon\psi'|\eta|^2\geq &\rho a^{ij}D_{ij} Z+\rho a^{ij}_k\left[w_{\alpha k} w_{\beta i j}\eta^{\al}\eta^{\be}+w_{\beta k} w_{\alpha ij}\eta^{\al}\eta^{\be}\right]\\
                &-2\rho a^{ij}c_jc_{\al}w_{,\be i}\eta^{\al}\eta^{\be}-2\rho a^{ij}c_ic_{\al}w_{,\be j}\eta^{\al}\eta^{\be}\\
                &-\epsilon\psi a^{ij}\left(|\eta|^2\right)_{ij}+\rho a^{ij}_kw_{ij}\left[Z_k-2c_{\alpha}w_{\beta k}\eta^{\al}\eta^{\be}-\epsilon\psi\left(|\eta|^2\right)_k\right]+\rho a^{ij}_{kl}w_{k\al}w_{l\be}w_{ij}\eta^{\al}\eta^{\be}\\
                &+\rho'w_{\al\be}a^{ij}w_{ij}\eta^{\al}\eta^{\be}+\rho' a_k^{ij}w_{\alpha}w_{k\beta}w_{ij}\eta^{\al}\eta^{\be}+\rho' a^k_{ij}w_{\beta}w_{k\alpha}w_{ij}\eta^{\al}\eta^{\be}\\
                &+ \left[pb_1(w)|\D w|^{p-2}+2b_2(w)\right](\nabla w\cdot \nabla Z)+p|\nabla w|^{p-2}\left[\left(b_1'(w)w_{\alpha}-c_{\alpha}b_1(w)\right)\D w\cdot\D w_{\beta}\eta^{\al}\eta^{\be}\right]\\
                &+p|\nabla w|^{p-2}\left[\left(b_1'(w)w_{\beta}-c_{\beta}b_1(w)\right)\D w\cdot\D w_{\alpha}\eta^{\al}\eta^{\be}\right]+\left[\left(b_2'(w)w_{\alpha}-c_{\alpha}b_2(w)\right)\D w\cdot\D w_{\beta}\eta^{\al}\eta^{\be}\right]\\
                &+\left[\left(b_2'(w)w_{\beta}-c_{\beta}b_2(w)\right)\D w\cdot\D w_{\alpha}\eta^{\al}\eta^{\be}\right]-\epsilon\psi\left[pb_1(w)|\D w|^{p-2}+2b_2(w)\right](\nabla w\cdot \nabla |\eta|^2)\\
                &\left[pb_1(w)|\D w|^{p-2}+2b_2(w)\right]\left(\nabla w_{\alpha}\cdot \nabla w_{\beta}\right)\eta^{\al}\eta^{\be}+p(p-2)b_1(w)|\nabla w|^{p-4}w_kw_lw_{k\al}w_{l\be}\eta^{\al}\eta^{\be}\\
                &+\left[b_1'(w)|\nabla w|^{p}+b_2'(w)|\nabla w|^2+b_3'(w)\right]w_{\alpha\beta}\eta^{\al}\eta^{\be}.
\end{aligned}
\end{equation}
Suppose that the minimum of the second derivatives of $w$ is taken along a direction $\overline{\alpha}$ at $(0,t_0)$. Then, $e_{\overline{\alpha}}$ is an eigen direction of the symmetric matrix $D^2w(0,t_0)$. Thus, at $(0,t_0)$, we have
\begin{equation*}
\begin{array}{c}
Z=w_{\overline{\alpha}\overline{\alpha}}+\varepsilon\psi(t_0)=0,\vspace{0.3cm}\\
w_{\overline{\alpha}\beta}=0 \qquad \mbox{if $\overline{\alpha}\neq\beta$ and $e_{\beta}\in \R^{n}$,}\vspace{0.3cm}\\
D^2Z\geq 0, \quad \La Z\geq 0, \quad \nabla Z=0 \quad \mbox{and} \quad Z_t\leq 0.
\end{array}
\end{equation*}
In addition, we get
\begin{equation*}
\left(|\eta|^2\right)_{ij}=2nc_{\overline{\alpha}}^2+\frac{c_{\overline{\alpha}}(c_i+c_j)}{2} \quad \mbox{and} \quad \nabla_{x}w\cdot\nabla_{x}|\eta|^2=2c_{\overline{\alpha}}w_{\overline{\alpha}} \qquad \mbox{at $(x_0,t_0)$}.
\end{equation*}
Since $a_{ij}$ is uniformly positive, by \eqref{eq-w-1-after-first-calculation}, we have
\begin{equation}\label{eq-w-1-after-first-calculation-after-reducing}
\begin{aligned}
-\epsilon\psi'\geq &2\rho a^{ij}_{\overline{\alpha}}w_{\overline{\alpha} i j}w_{\overline{\alpha}\overline{\alpha}}-4\rho c_{\overline{\alpha}}^2 a^{\overline{\alpha}\overline{\alpha}}w_{\overline{\alpha}\overline{\alpha}}-\epsilon\psi a^{ij}\left(|\eta|^2\right)_{ij}-2c_{\overline{\alpha}}\rho a^{ij}_{\overline{\alpha}}w_{ij}w_{\overline{\alpha}\overline{\alpha}}\\
&-\rho a^{ij}_kw_{ij}\epsilon\psi\left(|\eta|^2\right)_k+\rho a^{ij}_{\overline{\alpha}\overline{\alpha}}w_{ij}w_{\overline{\alpha}\overline{\alpha}}^2+\rho'a^{ij}w_{ij}w_{\overline{\al}\overline{\alpha}}+2\rho' w_{\overline{\alpha}}a_{\overline{\alpha}}^{ij}w_{ij}w_{\overline{\alpha}\overline{\alpha}}\\
&+2p|\nabla w|^{p-2}\left(b_1'(w)w_{\overline{\alpha}}-c_{\overline{\alpha}}b_1(w)\right)w_{\overline{\alpha}}w_{\overline{\alpha}\overline{\alpha}}+2\left(b_2'(w)w_{\overline{\alpha}}-c_{\overline{\alpha}}b_2(w)\right)w_{\overline{\alpha}}w_{\overline{\alpha}\overline{\alpha}}\\
                &-\epsilon\psi\left[pb_1(w)|\D w|^{p-2}+2b_2(w)\right](\nabla w\cdot \nabla |\eta|^2)+\left[pb_1(w)|\D w|^{p-2}+2b_2(w)\right]w_{\overline{\alpha}\overline{\alpha}}^2+p(p-2)b_1(w)|\nabla w|^{p-4}w_{\overline{\alpha}}^2w_{\overline{\al}\overline{\alpha}}^2\\
                &+\left[b_1(w)'|\nabla w|^{p}+b_2'(w)|\nabla w|^2+b_3'(w)\right]w_{\overline{\alpha}\overline{\alpha}}\\
                :=& -K\epsilon\psi-L\epsilon^2\psi^2
\end{aligned}
\end{equation}
where
\begin{equation*}
\begin{aligned}
K_1=&2\rho a^{ij}_{\overline{\alpha}}w_{\overline{\alpha} i j}-4\rho c_{\overline{\alpha}}^2 a^{\overline{\alpha}\overline{\alpha}}-2c_{\overline{\alpha}}\rho a_{ij}^{\overline{\alpha}}w_{ij}+\rho'a^{ij}w_{ij}+2\rho' w_{\overline{\alpha}}a_{\overline{\alpha}}^{ij}w_{ij}\\
&+2p|\nabla w|^{p-2}\left(b_1'(w)w_{\overline{\alpha}}-c_{\overline{\alpha}}b_1(w)\right)w_{\overline{\alpha}}+2\left(b_2'(w)w_{\overline{\alpha}}-c_{\overline{\alpha}}b_2(w)\right)w_{\overline{\alpha}}+b_1'(w)|\nabla w|^{p}+b_2'(w)|\nabla w|^2+b_3'(w)\\
&+a^{ij}\left(2nc_{\overline{\alpha}}^2+\frac{c_{\overline{\alpha}}(c_i+c_j)}{2}\right)+2c_{\overline{\alpha}}w_{\overline{\alpha}}\left[pb_1(w)|\D w|^{p-2}+2b_2(w)\right]
\end{aligned}
\end{equation*}
and
\begin{equation*}
L=-\rho a^{ij}_{\overline{\alpha}\overline{\alpha}}w_{ij}-\left(pb_1(w)|\D w|^{p-2}+2b_2(w)\right)-p(p-2)b_1(w)|\nabla w|^{p-4}w_{\overline{\alpha}}^2.
\end{equation*}
By the boundary condition, the minimum point $(0,t_0)$ appears far away from the boundary. Hence, we can choose a compact subset $\mathbb{S}$ in $\Omega$ such that
\begin{equation*}
(x_0(t),z_0(t))\in \mathbb{S} \qquad \mbox{and} \qquad w(x,t) : \mbox{bounded above and below by positive constants in $\mathbb{S}$}.
\end{equation*}
The equation \eqref{eq-w}, when restricted on $\mathbb{S}$, is nondegenerate. Therefore, the classical estimates for linear parabolic equations give
\begin{equation*}
|Dw(0,t)|\leq C\|w\|_{L^{\infty}(\mathbb{S})} \qquad \mbox{and} \qquad |D^2w(0,t)|\leq C\|w\|_{L^{\infty}(\mathbb{S})}
\end{equation*}
for constant $C>0$. Hence, the quantity $K$ and $L$ are under control, i.e., there exists a constant $M$ such that
\begin{equation*}
\left|K\right|,\,\,\left|L\right| \leq M \qquad \qquad \mbox{at $(0,t_0)$}.
\end{equation*} 
Hence, if we choose the function $\psi(t)$ having the property
\begin{equation*}
\psi_t(t)>|K|\psi+|L|\psi^2,
\end{equation*}
then, contradiction arises. Therefore, the time $t=t_0$ satisfying \eqref{eq-assumption-for-t-0-with-zero-level} doesn't exist. Thus,
\begin{equation*}
\inf_{x\in\Omega, s\in[0,t]}\inf_{e_{\beta}\in\R^n, |e_{\beta}|=1}w_{\beta\beta}>-\epsilon\sup_{0\leq s\leq t}\psi(s), \qquad \forall t>0.
\end{equation*}
Letting $\epsilon\to 0$, we can get a desired conclusion.
\end{proof}

\section{Geometric Property of Gelfand Problem}

\setcounter{equation}{0}
\setcounter{thm}{0}

In previous section, we discussed the positivity of the second derivatives of solutions for degenerated parabolic equation. It is a very useful tool for investigating the geometric properties of solutions to Gelfand's problem. We now address the long-time geometrical properties of solutions for the initial value problem with exponential growth:
\begin{equation}\label{eqn}
\left(e^{u}\right)_{t}=\La u+\lambda e^u \qquad \mbox{in $Q=\Omega\times (0,\infty)$}
\end{equation}
posed in a strictly convex bounded domain $\Omega\subset\R^n$ with
\begin{equation}\label{eq-boundary-data-u-0-by-phi-epsilon}
u=0 \qquad \mbox{on $\Omega$}, \qquad \qquad u>0 \mbox{in $\Omega$}
\end{equation}
and inital data 
\begin{equation}\label{eq-initial-data-u-0-by-phi-epsilon}
u(x,0)=u_0(x)\leq\vp(x)
\end{equation}
where $\vp$ is the minimal solution of the Gelfand's problem \eqref{eq-main-1}.\\ 
\indent In this section, we aim at providing the $f$-convexity of the minimal solution $\vp$ to the problem \eqref{eq-main-1} which also satisfies the boundary condition \eqref{boundary-condition-of-gelfand-with-PME}, i.e., we want to show that
\begin{equation*}
e^{-\frac{1}{2}\vp(x)} \quad : \quad \mbox{convex with respect to space variables}.
\end{equation*}
If we try to show the $f$-convexity of $\vp$ in \eqref{eq-main-1}, we can put $v=e^{-\frac{1}{2}\vp}$ and replace $\vp$ by $-2\log v$ in the equation. Then
\begin{equation*}
v\La v-|\nabla v|^2-\frac{\lambda}{2}=0.
\end{equation*}
By direct computation, we get
\begin{equation*}
v\La w_{\alpha\alpha}+2v_{\alpha}\La v_{\alpha}-2\nabla v\cdot\nabla v_{\alpha\alpha}-2\nabla v_{\alpha}\nabla v_{\alpha}=0.
\end{equation*}
Unfortunately, there are many terms, for example $\La v_{\alpha}$, which are out of control. We don't even have any information about $v_{\alpha\alpha}$. Hence, infering the geometric properties of $\vp$ from the equation \eqref{eq-main-1} directly is very hard.  However, according to the arguments in Lee and V\'azquez's paper, \cite{LV}, the geometric properties of solutions to the nonlinear elliptic problem can be obtained by the geometric properties of the solution to the corresponding problem with parabolic flow. To apply their arguments to the Gelfand's problem, we start by showing the relation between the solution $u$ of \eqref{eqn}-\eqref{eq-initial-data-u-0-by-phi-epsilon} and $\vp$ of \eqref{eq-main-1}.
\begin{lemma} [\bf Approximation lemma]\label{lemma-Approximation-Lemma}
Let $u(x,t)$ be a solution of \eqref{eqn} and let $\vp$ be the minimal solution of \eqref{eq-main-1}. Then, we have the following properties: For any sequence $\{t_n\}_{n=1}^{\infty}$ with $t_n\to\infty$, we have a subsequence $\{t_{n_k}\}_{k=1}^{\infty}$ such that
\begin{equation}
\lim_{k\to\infty} |u(x,t_{n_k})-\vp(x)|\to 0
\label{eq-pme-approx}
\end{equation}
uniformly in  compact subset of $\Omega$.
\end{lemma}
\begin{proof}
Define the functional $F(\phi)$ by
\begin{equation*}
F(\phi)=-\int_{\Omega}\left(\frac{1}{2}\phi\La\phi+\lambda e^{\phi}\right)\,dx
\end{equation*}
and let $g(t)=F(u(\cdot,t))$. Then a simple computation yields
\begin{equation}\label{eq-negativity-of-g-prime-1}
g'(t)=-\lambda\int_{\Omega}e^{u(x,t)}\left[u_t(x,t)\right]^2\,dx\leq 0.
\end{equation}
Uniformly ellipticity of the coefficients in \eqref{eqn} show that $\int_{\Omega}e^{u(x,t)}\,dx$ is bounded for all $t\geq 0$. Hence, for some constant $M>0$,
\begin{equation}\label{eq-bounded-below-of-g-323}
g(t)>-M \qquad \forall t>0.
\end{equation}
Therefore, by \eqref{eq-negativity-of-g-prime-1} and \eqref{eq-bounded-below-of-g-323}, $\lim_{t\to\infty}g(t)$ exists and $g'(t)\to 0$. Hence, for any sequence of times $\{t_n\}_{n=1}^{\infty}$, $t_n\to\infty$, $g'(t_n)\to 0$.\\
\indent Observe that the equation \eqref{eqn} is uniformly parabolic in $\Omega$. Thus, by the comparison principle, there exists a uniform constant $C$ such that
\begin{equation}\label{eq-comparison-principle-of-u-and-vp}
0\leq u(x,t)\leq \vp(x)\leq C.
\end{equation}
Moreover, by the Schauder estimates for parabolic equation \cite{LSU}, the sequence $u(\cdot,t_n)$ is equi-H\"older continuous on every compact subset $K$ of $\Omega$. Hence, by the Ascoli Theorem, the  sequence $\{u(\cdot,t_n)\}$ has a subsequence $\{u(\cdot,t_{n_k})\}$ that converges to some function $h$ uniformly on every compact subset of $K$ which is non-trivial. The readers can easily check the non-triviality of $h$.\\
Multiplying equation \eqref{eqn} by any test function $\eta\in C_{0}^{\infty}(\Omega)$ and integrate in space. Then for each $t_{n_k}$,
\begin{equation}\label{eq-integral-form-of-u-for-behavioru-limit-h}
\lambda\int_{K}e^uu_t\eta\,dx=\int_{K}u\La\eta\,dx+\lambda\int_{K}e^{u}\eta\,dx.
\end{equation} 
Since the absolute value of the left hand side of \eqref{eq-integral-form-of-u-for-behavioru-limit-h} is bounded above by
\begin{equation}\label{eq-divided-the-left-hand-term-by-two-with-coverging-zero-term}
\lambda\left(\int_{\Omega}\eta^2e^u\,dx\right)^{\frac{1}{2}}\left(\int_{\Omega}e^u\left(u_t\right)^2\,dx\right)^{\frac{1}{2}}
\end{equation}
and the second term of \eqref{eq-divided-the-left-hand-term-by-two-with-coverging-zero-term} has limit zero as $t_{n_k}\to\infty$, we  get in the limit $t_{n_k}\to\infty$
\begin{equation*}
0=\int_{K}h\La\eta\,dx+\lambda\int_{K}e^{h}\eta\,dx,
\end{equation*} 
which is weak formulation of the equation
\begin{equation}\label{eq-weak-formulation-of-gelfand-s-problem-394}
\La h+\lambda e^h=0 \qquad \mbox{in $K$}.
\end{equation}
By the arbitrary choice of a compact subset $K$ in $\Omega$, \eqref{eq-weak-formulation-of-gelfand-s-problem-394} holds in $\Omega$. Therefore, $h$ is a weak solution of the Gelfand's problem which is smaller than the minimal solution $\vp$ because of \eqref{eq-comparison-principle-of-u-and-vp}. Since $\vp$ is the minimal solution of Gelfand's problem, $u(x,t_{n_k})$ converges uniformly on every compact subset of $\Omega$ to $\vp$ as $t_{n_k}\to\infty$ and the lemma follows.
\end{proof}

\subsection{geometric property}

We first establish some estimates for the solution $u$ of \eqref{eqn} which plays an important role for the geometric properties of solution $u$ on the boundary. 

\begin{lemma}\label{lem-estimate-of-mixed-second-derivatives-by-first-normal-direction}
Let $u\in C^2\left(\Omega)\times\left(0,\infty\right)\right)$ be a solution of  \eqref{eqn}-\eqref{eq-initial-data-u-0-by-phi-epsilon}. Then, there exists a constant $C>0$ such that
\begin{equation*}
|u_{\tau_{k}\nu}(x,t)|\leq C|u_{\nu}(x,t)|, \qquad \forall x\in\partial\Omega \quad \left(k=1, \cdots, n-1\right)
\end{equation*}
where $\nu$ and $\tau_k$ are normal and tangential directions to $\partial\Omega$ at $(x,t)$, respectively.
\end{lemma}
\begin{proof}
Without loss of generality, we may assume that $x=0$ and outer normal direction $e_{\nu}=e_n$. For any $1\leq k\leq n-1$, we consider the directional derivative 
\begin{equation*}
\partial_{T_k}u=P(x)\cdot\nabla u, \qquad \left(P(x)=(p_1(x),\cdots,p_n(x))\right)
\end{equation*}
which is the same as a tangential derivative on the boundary $\partial\Omega$. By the strictly convexity of domain $\Omega$, we can find the largest circle of radius $R=R_k$ that touches the domain $\Omega$ from inside at $x=0$ in the plane generated by two directions $e_{\tau_k}$ and $e_{\nu}$. Note that the reciprocal of $R$ lies on between principle curvatures at $x=0$. Moreover, $\left(R-x_n\right)u_k+x_ku_n$ is a tangential derivative on the circle. Thus, the directional derivative $\partial_{T_k}u$ can be expressed in the form
\begin{equation*}
\partial_{T_k}u(x,t)=\left(R-x_n+O(|x|^2)\right)u_k(x,t)+\left(x_k+O(|x|^2)\right)u_n(x,t)
\end{equation*}
near the point $x=0$. \\
\indent For a constant $0<\rho<1$ and a time $t_0>0$, let $B_{\rho}=B_{\rho}(0)$ be the ball of radius $\rho$ centered at $0$ and define the new functions $h_{\pm}$ and $H_{\pm}$ by
\begin{equation}\label{eq-aligned-test-function-h-and-v-for-comparison-3245}
\begin{aligned}
h_{\pm}&=\partial_{T_k}u\pm\sum_{l=1}^{n-1}\left(\partial_{T_l}u\right)^2,\\
H_{\pm}&=\eta^2h_{\pm}\pm\mu x_n^2, 
\end{aligned}
\end{equation}
where $\eta(x,t)=\left(\rho^2-|x|^2\right)\left(t-t_0+\rho^2\right)$. Then, it can be easily checked that
\begin{equation}\label{eq-boundary-comparison-linear-and-v-2}
\left|H_{\pm}(x,t)\right|\leq -\frac{\mu\rho^2}{2} x_n .
\end{equation}
on $\left\{\partial(\Omega\cap B_{\rho})\right\}\times[t_0-\rho^2,t_0]$ and $\left\{\Omega\cap B_{\rho}\right\}\times\{t=t_0-\rho^2\}$\\
\indent Let us now show that $L H_{+}\geq 0$ and $L H_{-}\leq 0$ in $\Omega\cap B_{\rho}$ where $L$ is defined by
\begin{equation*}
Lf=\La f+\lambda e^uf-f_t.
\end{equation*}
By direct computation, we can find a constant $\rho_1>0$ such that 
\begin{equation*}
\sum_{l=1}^{n-1}\left|\partial_{T_l}u\right|^2\geq \frac{R^2}{4}\sum_{l=1}^{n-1}\left|\nabla u_{l}\right|^2-C_1\left(1+|\nabla u|^2\right) \qquad \forall 0<\rho\leq \rho_1
\end{equation*}
for some constant $C_1>0$. Thus, we have, for some constant $c>0$,
\begin{equation*}
\begin{aligned}
LH_{+}&\geq\eta^2\left(\La h_{+}+\lambda e^{u}h_{+}-\left(h_{+}\right)_t\right)+2h_{+}\left(\eta\La\eta+|\nabla\eta|^2-\eta\eta_t\right)\\
&\quad +4\eta\nabla\eta\cdot\nabla h_{+}+2\mu\\
&\geq 2\eta^2\sum_{l=1}^{n-1}\left|\nabla\partial_{T_l}u\right|^2\\
&\quad -c\eta^2\bigg[1+2\left(R|\nabla_{x'}u|+\rho|u_n|\right)\bigg]\left[\sum_{l=1}^{n-1}\left\{|\nabla u|+\rho\left(\sum_{i=1}^{n}|u_{li}|+\sum_{i=1}^{n-1}|u_{in}|+|\La_{x'}u|+\lambda e^u+|u_t|\right)\right\}\right]\\
&\qquad -\lambda \eta^2e^u\left(R^2|\nabla_{x'}u|^2+2R\rho|\nabla_{x'}u||u_n|+\rho^2|u_n|\right)\\
&\qquad -4n(n+3)\rho^6\left(R|\nabla_{x'}u|+\rho|u_n|+R^2|\nabla_{x'}u|^2+2R\rho|\nabla_{x'}u||u_n|+\rho^2|u_n|^2\right)\\
&\qquad -16\rho^3\eta\left(R|\nabla_{x'}u|+\rho|u_n|\right)\sum_{l=1}^{n-1}\left|\nabla\left(\partial_{T_l}u\right)\right|\\
&\qquad -8\rho^4\eta\left[\sum_{i=1}^{n-1}|u_{ni}|+|\La_{x'}u|+\lambda e^u+|u_t|\right]-8R\rho^3\eta\sum_{i=1}^{n-1}|u_{ki}|-8\rho^3\eta|\nabla u|+2\mu\\
&\geq\left[\eta^2\sum_{l=1}^{n-1}\left|\nabla\partial_{T_l}u\right|^2-16\rho^3\eta\left(R|\nabla_{x'}u|+\rho|u_n|\right)\sum_{l=1}^{n-1}\left|\nabla\left(\partial_{T_l}u\right)\right|+\frac{\mu}{2}\right]\\
&\qquad +\left[\frac{R^2\eta^2}{8}\sum_{l=1}^{n-1}\left|\nabla u_l\right|^2-3c\rho\eta^2\bigg[1+2\left(R|\nabla_{x'}u|+\rho|u_n|\right)\bigg]\sum_{l=1}^{n-1}\left|\nabla u_l\right|+\frac{\mu}{2}\right]\\
&\qquad +\left[\frac{R^2\eta^2}{8}\sum_{l=1}^{n-1}\left|\nabla u_l\right|^2-8\rho^3\eta\left(R+\rho\right)\sum_{l=1}^{n-1}\left|\nabla u_l\right|+\frac{\mu}{2}\right]\\
&\qquad +\frac{\mu}{2}-2\rho^6\left[K_0+K_1\rho+K_2\rho^2+K_3\rho^3+K_4\rho^4\right]
\end{aligned}
\end{equation*}
where 
\begin{equation*}
K_0=4n(n+3)\left(R|\nabla_{x'}u|+R^2|\nabla_{x'}u|^2\right)
\end{equation*}
\begin{equation*} 
K_1=4\left[n(n+3)\left(|u_n|+2R|\nabla_{x'}u||u_n|\right)+2|\nabla u|\right],
\end{equation*}
\begin{equation*}
K_2=\frac{C_1}{2}\left(1+|\nabla u|^2\right)+4n(n+3)|u_n|^2+\lambda R^2e^u|\nabla_{x'}u|+8\left(\lambda e^u+|u_t|\right)+c(n-1)(1+2R|\nabla_{x'}u|)|\nabla u|,
\end{equation*}
\begin{equation*}
K_3=c(n-1)\left[|u_n||\nabla u|+\left(1+2R|\nabla_{x'}u|\right)\left(\lambda e^u+|u_t|\right)\right]+2R\lambda\rho e^u|\nabla_{x'}u||u_n|
\end{equation*}
and
\begin{equation*} 
K_4=c(n-1)|u_n|\left(\lambda e^u+|u_t|\right)+\lambda e^u|u_n|.
\end{equation*}
We now choose a constant $\rho_2$ such that , $\forall 0<\rho\leq \rho_2$,
\begin{equation*}
\rho|\nabla u|^2\leq 1, \qquad 2^7\left(R|\nabla_{x'}u|+\rho|u_n|\right)^2<|u_n(0,t_0)|,
\end{equation*}
\begin{equation*}
\frac{36c^2\rho^4}{R^2}\bigg[1+2\left(R|\nabla_{x'}u|+\rho|u_n|\right)\bigg]^2\leq |u_n(0,t_0)|, \qquad 2^8\left(1+\frac{\rho}{R}\right)^2\leq 2^9
\end{equation*}
and
\begin{equation*}
K_0+K_1\rho+K_2\rho^2+K_3\rho^3+K_4\rho^4<\frac{|u_n(0,t_0)|}{4}
\end{equation*}
in $\left\{\Omega\cap B_{\rho}\right\}\times[t_0-\rho^2,t_0]$. Thus, for $0<\rho<\min\left\{\rho_1,\rho_2\right\}$,
\begin{equation}\label{eq-sub-and-super-solution-of-L-v-+-and-v--}
L\left(H_+\right)\geq 0 \qquad \qquad \mbox{if $\mu\geq \max\{|u_n(0,t_0)|\rho^6,2^9\rho^6\}=2C_1|u_n(0,t_0)|\rho^6$}.
\end{equation}
Hence, by \eqref{eq-boundary-comparison-linear-and-v-2} and \eqref{eq-sub-and-super-solution-of-L-v-+-and-v--}, 
\begin{equation*}
H_{+}(x,t)\leq -C_1|u_n(0,t_0)|\rho^8x_n \qquad \mbox{on $\Omega\cap B_{\rho}\times(t_0-\rho^2,t_0]$}.
\end{equation*}
This immediately implies that
\begin{equation*}
h_{+}(x,t)\leq C_1u_{n}(0,t_0)x_n \qquad \mbox{in $\Omega\cap B_{\frac{\rho}{2}}\times\left(t_0-\frac{\rho^2}{4},t_0\right]$}
\end{equation*}
Similarly, we can also show that
\begin{equation*}
h_{-}(x,t)\geq -C_2u_{n}(0,t_0)x_n \qquad \mbox{in $\Omega\cap B_{\frac{\rho}{2}}\times\left(t_0-\frac{\rho^2}{4},t_0\right]$}
\end{equation*}
for some constant $C_2>0$. Therefore
\begin{equation*}
-\overline{C}Ru_n(0,t_0)x_n\leq v_{-}(x,t)\leq v_{+}(x,t)\leq \overline{C}Ru_{n}(0,t_0)x_n \qquad \mbox{in $\Omega\cap B_{\frac{\rho}{2}}\times\left(t_0-\frac{\rho^2}{4},t_0\right]$}.
\end{equation*}
for constant $\overline{C}=\frac{1}{R}\max\{C_1,C_2\}$. Taking the normal derivative $\partial_n$ at $(0,t_0)$, we obtain
\begin{equation*}
\left|u_{kn}(0,t_0)\right|=\left|\frac{\left(v_+\right)_n(0,t_0)}{R}\right|\leq \overline{C}|u_n(0)|, \qquad \forall 1\leq k\leq n-1
\end{equation*}
and the lemma follows.
\end{proof}

Next, coming to our subject, we have the following result about {\sl preservation of f-
convexity,} which is easy but allows to present the basic technique. Our geometrical
results will be derived under the extra assumption that $\Omega$
is strictly convex.
\begin{lemma}\label{lemma-boundary-estimate-for-power-convexity}
Let $\Omega$ be a strictly convex bounded subset in $\R^n$ and assume that $u$ is a solution of \eqref{eqn}-\eqref{eq-initial-data-u-0-by-phi-epsilon} with the boundary condition \eqref{boundary-condition-of-gelfand-with-PME}. Let $u=-2\log w$. Then, for every $t>0$, as $x\to x_0\in\partial\Omega$
\begin{equation}\label{eq-aligned-on-the-boundary-estimate-of-second-derivatives-of-w-all-direction}
w_{\alpha\alpha}(x,t)=\frac{1}{2}e^{-\frac{1}{2}u}\left(\frac{1}{2}u_{\alpha}^2-u_{\alpha\alpha}\right)\geq \delta_0>0, \qquad \left(\alpha=e_{\alpha}\in\R^n,\quad \mbox{and} \quad \left|e_{\alpha}\right|=1\right)
\end{equation}
for a uniform constant $\delta_0$ depending on  $\partial\Omega$.
\end{lemma}

\begin{proof}
By direct computation, we have
\begin{equation*}
w_{\alpha\alpha}(x,t)=\frac{1}{2}e^{-\frac{1}{2}u}\left(\frac{1}{2}u_{\alpha}^2-u_{\alpha\alpha}\right), \qquad \mbox{on $\partial\Omega$}.
\end{equation*}
If $e_{\alpha}=\tau$, a
tangent direction at $x_0$ to $\partial\Omega$, then
$u_{\alpha}=0$. Hence, we need to estimate
$u_{\alpha\alpha}$. For this, we use the fact
that $\partial \Omega$ is strictly convex. Without loss of generality, we may assume that $x_0=0$ and the tangent plane is
$x_{n}=0$. We may also assume that the boundary is given locally by
the equation $x_{n}=f(x')$, and $x'=(x_1,\cdots,x_{n-1})$. We introduce
the change of variables
\begin{equation}\label{eq-change-of-variables-for-flattening-the-boundary}
\tilde{x}_{\alpha}=x_{\alpha},\qquad \tilde{x}_{n}=x_{n}-f(x'),
\qquad
g(\tilde{x}',\tilde{x}_{n},t)=u(x',x_{n},t).
\end{equation}
Then along tangent directions we have
\begin{equation*}
\begin{aligned}
u_{\alpha\alpha}(x',x_{n},t)=&g_{\alpha\alpha}(\tilde{x}',\tilde{x}_{n},t)-2g_{\alpha\tilde{x}_{n}}(\tilde{x}',\tilde{x}_{n},t)f_{\alpha}(x')\\
&+g_{\tilde{x}_{n}\tilde{x}_{n}}(\tilde{x}',\tilde{x}_{n},t)(f_{\alpha}(x'))^2-g_{\tilde{x}_{n}}(\tilde{x}',\tilde{x}_{n},t)f_{\alpha\alpha}(x').
\end{aligned}
\end{equation*}
Since $f_{\alpha}(0)=0$ and $f_{\alpha\alpha}(0)$ are nonzeros along
all tangent directions, we get
\begin{equation}\label{eq-tangential}
u_{\alpha\alpha}(0,0,t)=-g_{\tilde{x}_{n}}(0,0,t)f_{\alpha\alpha}(0)=-u_{x_{n}}(0,0,t)f_{\alpha\alpha}(0).
\end{equation}
Since $u$ is a solution to a parabolic equation with uniformly elliptic coefficients, by the Hopf's lemma for the classical partial differential equation, $0<c_0\leq |\nabla u|$. Hence
\begin{equation*}
u_{\tau\tau}(x_0,t)=-u_{\nu}(x_0,t)f_{\tau\tau}(0)<0
\end{equation*}
and
\begin{equation}\label{eq-aligned-on-the-boundary-estimate-of-second-derivatives-of-w-tangential}
\begin{aligned}
w_{\tau\tau}(x,t)=\frac{1}{2}e^{-\frac{1}{2}u}\left(\frac{1}{2}u_{\tau}^2-u_{\tau\tau}\right)> \frac{c_0f_{\tau\tau}(0)}{2} \qquad \mbox{as $x\to x_0\in\partial\Omega$}.
\end{aligned}
\end{equation}
Let $e_{\alpha}=\nu$. On $\partial\Omega$
\begin{equation*}
0=\lambda\left(e^u\right)_t=\La u+\lambda e^u=\La u+\lambda.
\end{equation*}
Thus, we have
\begin{equation*}
\frac{1}{2}u_{\nu}^2-u_{\nu\nu}=\frac{1}{2}u_{\nu}^2+\lambda+\sum_{i=1}^{n-1}u_{\tau_i\tau_i}=\frac{1}{2}u_{\nu}^2+\lambda+(n-1)u_{\nu}H(\partial\Omega) \qquad \mbox{on $\partial\Omega$},
\end{equation*}
where $H(\partial\Omega)$ is the mean curvature of $\partial\Omega$ at $x_0$. By the boundary condition \eqref{boundary-condition-of-gelfand-with-PME}, we can obtain
\begin{equation}\label{eq-aligned-on-the-boundary-estimate-of-second-derivatives-of-w-normal}
\begin{aligned}
w_{\nu\nu}(x,t)=\frac{1}{2}e^{-\frac{1}{2}u}\left(\frac{1}{2}u_{\nu}^2-u_{\nu\nu}\right)=\frac{1}{2}\left(\frac{1}{2}u_{\nu}^2+\lambda+(n-1)u_{\nu}H(\partial\Omega)\right)>0
\qquad \mbox{on $\partial\Omega$}.
\end{aligned}
\end{equation}
\indent Finally, we check the case that the minimum of the second derivatives of $w$ occurs along a general direction $e_{\alpha}$ at $x_0$. Since the outer normal direction $e_{\nu}$ is vertical to the tangent plane, we can express a general direction $e_{\alpha}$ by
\begin{equation}\label{eq-general-for-directions-1}
\begin{aligned}
e_{\alpha}=k_1e_{\alpha_{\tau}}+k_2e_{\nu}, \qquad \left(k_1^2+k_2^2=1\right).
\end{aligned}
\end{equation}
where $e_{\alpha_{\tau}}$ is a direction contained on the tangent plane to the graph of $u$ at $(0,t)$. Hence, the second derivatives of $w$ can be written in the form
\begin{equation*}
\begin{aligned}
w_{\alpha\alpha}&=k_2^2w_{\nu\nu}+2k_1k_2w_{\alpha_{\tau}\nu}+k_1^2w_{\alpha_{\tau}\alpha_{\tau}}\\
&=\frac{1}{2}e^{\frac{1}{2}u}\left[k_2^2\left(\frac{1}{2}u_{\nu}^2-u_{\nu\nu}\right)+2k_1k_2\left(\frac{1}{2}u_{\alpha_{\tau}}u_{\nu}-u_{\alpha_{\tau}\nu}\right)+k_1^2\left(\frac{1}{2}u_{\alpha_{\tau}}^2-u_{\alpha_{\tau}\alpha_{\tau}}\right)\right].
\end{aligned}
\end{equation*}
Thus, at $x=x_0$, we have
\begin{equation}\label{eq-second-derivatives-in-general-direction-of-w}
w_{\alpha\alpha}=\frac{1}{2}e^{\frac{1}{2}u}\left[\frac{k_2^2}{2}u_{\nu}^2-k_2^2u_{\nu\nu}-2k_1k_2u_{\alpha_{\tau}\nu}-k_1^2u_{\alpha_{\tau}\alpha_{\tau}}\right].
\end{equation}
 By Lemma \ref{lem-estimate-of-mixed-second-derivatives-by-first-normal-direction}, there exists some constant $C_1>0$ such that
\begin{equation}\label{eq-upper-bound-of-mixed-second-derivatives-by-normal-direction}
\left|u_{\alpha_{\tau}\nu}(0,t)\right|=\left|v_{\alpha_{\tau}\nu}(0,t)\right|\leq C_1\left|v_{\nu}(0,t)\right|=C_2\left|u_{\nu}(0,t)\right|
\end{equation}
Combining \eqref{eq-upper-bound-of-mixed-second-derivatives-by-normal-direction}  with \eqref{eq-second-derivatives-in-general-direction-of-w}, we can get
\begin{equation*}
w_{\alpha\alpha}=\frac{e^{-\frac{1}{2}u}}{2}\left[k_2^2\left(\frac{1}{2}u_{\nu}^2+\lambda+(n-1)H(\partial\Omega)u_{\nu}\right)-C_1k_1k_2|u_{\nu}|+k_1^2f_{\alpha_{\tau}\alpha_{\tau}}u_{\nu}\right] \qquad \mbox{at $x=x_0$}.
\end{equation*}
Therefore, if $u$ satisfies the boundary condition \eqref{boundary-condition-of-gelfand-with-PME} for 
\begin{equation*}
K\geq \frac{C_1^2}{4\left|f_{\alpha_{\tau}\alpha_{\tau}}(0)\right|},
\end{equation*}
then the result \eqref{eq-aligned-on-the-boundary-estimate-of-second-derivatives-of-w-all-direction} holds for all $e_{\alpha}\in\R^n$, $(\left|e_{\alpha}\right|=1)$ and the lemma follows.
\end{proof}

\begin{thm}
Let $\Omega$ be a convex  bounded domain and let $u_0\geq 0$ be a continuous and bounded initial function which satisfies \eqref{eq-initial-data-u-0-by-phi-epsilon}. Then, the solution $u$ of \eqref{eqn}-\eqref{eq-initial-data-u-0-by-phi-epsilon} with the boundary condition \eqref{boundary-condition-of-gelfand-with-PME} is $f$-convex in the space variable for all $t\geq 0$, i.e., 
\begin{equation}\label{eq-expression-of-power-convexity-(-M)}
D^2\left(e^{-\frac{1}{2}u}\right)\geq 0.
\end{equation}
\end{thm}

\begin{proof}
Let $w=e^{-\frac{1}{2}u}$ Then, the new function $w$ satisfies
\begin{equation}\label{eq-for-convexity-of-gelfand-problem-after-changing}
\lambda w_t=w^2\La w-w|\nabla w|^2-\frac{\lambda}{2}w.
\end{equation}
Hence, $w$ is a solution to the equation \eqref{eq-w} with $\rho(w)$, $a_{ij}(\nabla w)$, $b_1(w)$, $b_2(w)$ and $b_3(w)$ being replaced by $w^2$,
$I_n$, $0$, $-w$ and $-\frac{\lambda}{2}w$ respectively. Here, $I_n$ is the $n\times n$ identity matrix. \\
\indent We now are going to check that the condition \textbf{I.1} and \textbf{I.2} given in Section \ref{Section-Convexity-for-degenerate-equation} hold for the coefficients in the equation \eqref{eq-for-convexity-of-gelfand-problem-after-changing}. It is trivial that
\begin{equation*}
b_1(w)=0,\,\, b_2(w)=-w,\,\, b_3(w)=-\frac{\lambda}{2}w\,\,\mbox{\textbf{:}} \quad \mbox{convex}.
\end{equation*}
In addition, by direct computation, it can be easily shown that
\begin{equation*}
\rho''-\frac{(\rho')^2}{2\rho}=2-\frac{4w^2}{2w^2}=2-2=0.
\end{equation*}
Hence, the equation \eqref{eq-for-convexity-of-gelfand-problem-after-changing} have the condition \textbf{I} and \textbf{II}. On the other hand, Lemma \ref{lemma-boundary-estimate-for-power-convexity} tells us that the solution of the equation \eqref{eq-for-convexity-of-gelfand-problem-after-changing} with zero boundary condition is convex on the boundary. Therefore, by Lemma \ref{lemma-convexity-for-the-general-case-0}, it is also convex in the interior of domain $\Omega$ and the lemma follows. 
\end{proof}

\begin{corollary}
If $\Omega$ is convex, the the minimal stationary profile $\vp(x)$ of $u(x,t)$ is $f$-convex, i.e., $D^2\left[e^{-\frac{1}{2}\vp(x)}\right]\geq 0$.
\end{corollary}
\begin{proof}
Take the initial data as before. By the asymptotic result, Lemma \ref{lemma-Approximation-Lemma}, we have uniform convergence between $\vp(x)$ and $u(x,t)$. Hence, the conclusion follows.
\end{proof}

Note that $\vp$ satisfies the equation
\begin{equation*}
\La\vp+\lambda e^{\vp}=0 \qquad \mbox{in $\Omega$}.
\end{equation*}
Then $\overline{w}=f(\vp)^2=e^{-\vp}$ satisfies
\begin{equation*}
\overline{w}\La\overline{w}-\left|\nabla \overline{w}\right|^2-\lambda\overline{w}=0 \qquad \mbox{in $\Omega$}.
\end{equation*}
It has the similar form to the equation in (3.12) of the paper \cite{LV}. Hence, by an argument similar to the proof of Lemma 3.6 in \cite{LV}, we have the following strictly $f$-convexity for the minimal solution $\vp$ of Gelfand's problem.

\begin{lemma}[Strictly $f$-convexity]
If $\Omega$ is smooth and strictly convex, then the minimal solution $\vp(x)$ is strictly $f$-convex: there exists a constant $c_1>0$ such that
\begin{equation*}
D^2f(\vp)=D^2\left(e^{-\frac{1}{2}\vp}\right)\geq c_1\textbf{I}.
\end{equation*}
The constant $c_1$ depends only on the shape of $\Omega$.
\end{lemma}

\section{Boundary Condition for the Geometric Properties}
\setcounter{equation}{0}
\setcounter{thm}{0}

Through the previous sections, we investigated the $f$-convexity of the minimal solution of Gelfand's problem, \eqref{eq-main-1}, under the assumption that the minimal solution has the property \eqref{boundary-condition-of-gelfand-with-PME} which we call \textbf{boundary condition}. As mentioned before, if we have the boundary condition removed from our setting, it is very difficult to get the $f$-convexity of the minimal solution of \eqref{eq-main-1} on the boundary. Thus, we had no choice but to add the boundary condition \eqref{boundary-condition-of-gelfand-with-PME} for the result. However, the problem is that we couldn't guarantee the existence of the minimal solution of \eqref{eq-main-1} having the boundary condition \eqref{boundary-condition-of-gelfand-with-PME}. Therefore, we need to check whether it is possible to satisfy conditions, \eqref{eq-main-1} and \eqref{boundary-condition-of-gelfand-with-PME}, simultaneously. Otherwise, it is meaningless to apply our geometric results in physical models.\\
\indent Before we finish this work, we will introduce the existence of solution to Gelfand's problem \eqref{eq-main-1} having the boundary condition \eqref{boundary-condition-of-gelfand-with-PME}  in this last section of the paper. We are now ready to state and prove the main result in this section.

\begin{lemma}\label{lem-boundary-condition-on-a-ball-}
Let $\Omega$ be a ball with radius $r>0$, i.e., $\Omega=B_{r}$, and let $\lambda^{\ast}(B_r)$ be the extremal value of the Gelfand's problem \eqref{eq-main-1} in a ball $B_r$. Then, there exists $\overline{\lambda}=\overline{\lambda}(B_r)$ with $0<\overline{\lambda}\leq \lambda^{\ast}(B_r)$ such that 
\begin{equation}\label{eq-boundary-condition-of-general-solution-3}
G(\vp_{\lambda},\lambda,B_r)>0 \qquad \forall 0<\lambda<\overline{\lambda}
\end{equation}
where $\vp_{\lambda}$ is the minimal solution of the Gelfand's problem in $\Omega=B_r$. Here, the constant $\overline{\lambda}$ is depending on the radius $r$ and, by scaling arguments, is equal to $\frac{\overline{\lambda}(B_1)}{r^2}$, i.e., $\overline{\lambda}(B_r)=\frac{\overline{\lambda}(B_1)}{r^2}$.
\end{lemma}

\begin{proof}
For the convenient, we assume in this proof that $B_r=B_r(0)$ a ball of radius $r$ which is centered at $0$. For the minimal solution $\vp_{\lambda}$, let $M_{\lambda,r}$ be such that
\begin{equation}\label{eq-definition-of-M-ladmbda-r-in-a-ball}
M_{\lambda,r}=\left\|\vp_{\lambda}\right\|_{L^{\infty}(B_r)}.
\end{equation}
Consider the scaled function
\begin{equation*}
\overline{\vp}(x)=\vp_{\lambda}(rx).
\end{equation*}
Then, $\overline{\vp}$ is a function defined on the ball $B_1$ and satisfies
\begin{equation*}
\La\overline{\vp}+\lambda r^2e^{\overline{\vp}}=0 \qquad \mbox{in $B_1$}.
\end{equation*}
Thus, by \eqref{eq-definition-of-M-ladmbda-r-in-a-ball},
\begin{equation*}
M_{\lambda r^2,1}=\left\|\overline{\vp}\right\|_{L^{\infty}(B_1)}=\left\|\vp_{\lambda}\right\|_{L^{\infty}(B_r)}=M_{\lambda,r}.
\end{equation*}
Note that the minimal solution $\vp_{\lambda}$ converges to $0$ uniformly as $\lambda\to 0$. Thus, we can choose a constant $\lambda_1>0$ such that
\begin{equation}\label{eq-condition-of-special-constant-lambda-1-30}
e^{M_{\lambda,r}}<\frac{n}{n-1}\qquad \forall 0<\lambda<\lambda_1.
\end{equation}
Define the functions $\theta_{\lambda}, \overline{\theta}_{\lambda}: B_r\to\R$ by $\theta_{\lambda}(x)=\frac{\lambda}{2n}\left(r^2-|x|^2\right)$, $\overline{\theta}_{\lambda}(x)=\frac{\lambda e^{M_{\lambda,r}}}{2n}\left(r^2-|x|^2\right)$, respectively. Then, by a direct computation
\begin{equation*}
\La\theta_{\lambda}(x)=-\lambda \qquad \mbox{and} \qquad \La\overline{\theta}_{\lambda}(x)=-\lambda e^{M_{\lambda,r}}.
\end{equation*} 
Then, by the maximum principle for the super-harmonic function, we have 
\begin{equation*}
\theta_{\lambda}\leq \vp_{\lambda}\leq \overline{\theta}_{\lambda} \quad  \mbox{in $B_r$} \qquad \mbox{and} \qquad \theta_{\lambda}=\vp_{\lambda}=\overline{\theta}_{\lambda} \quad \mbox{on $\partial B_r$}.
\end{equation*}
This immediately implies that
\begin{equation}\label{eq-compare-of-nomal-derivatives-of-solution-with-theta}
-\frac{\lambda e^{M_{\lambda,r}}r}{n}\leq {\vp}_{\lambda,\nu}\leq -\frac{\lambda r}{n} \qquad \mbox{on $\partial B_r$}
\end{equation} 
where $\nu$ is the outer normal vector at the boundary $\partial B_{r}$.\\
\indent Since the minimal solution $\vp_{\lambda}$ is unique and is defined in a ball, it is radially symmetric. Thus, 
\begin{equation*}
\vp_{\lambda,\tau_i\nu}=0 \quad \mbox{on $\partial B_r$} \quad \forall i=1,\cdots, n \qquad \Rightarrow \qquad K=0 
\end{equation*}
where $\tau_i$, $(i=1,\cdots,n)$, are tangential directions and $K$ is the constant in the boundary condition \eqref{boundary-condition-of-gelfand-with-PME}. Combining this with \eqref{boundary-condition-of-gelfand-with-PME}, we have
\begin{equation*}
G(\vp_{\lambda},\lambda,B_r)=\frac{1}{2}\vp_{\lambda,\nu}^2+\frac{n-1}{r}\vp_{\lambda,\nu}+\lambda=\frac{1}{2}\left(\vp_{\lambda,\nu}^2+\frac{2(n-1)}{r}\vp_{\lambda,\nu}+2\lambda\right)=:\frac{1}{2}\overline{G}.
\end{equation*}
Observe that $\overline{G}$ can be referred to as a quadratic polynomial with respect to $\vp_{\lambda}$. Suppose that $\lambda<\frac{(n-1)^2}{2r^2}$. Then, $\overline{G}$ is strictly positive if
\begin{equation*}
\vp_{\lambda,\nu}<-\frac{n-1}{r}-\sqrt{\frac{(n-1)^2}{r^2}-2\lambda} \qquad \mbox{or} \qquad -\frac{n-1}{r}+\sqrt{\frac{(n-1)^2}{r^2}-2\lambda}<\vp_{\lambda,\nu}<0.
\end{equation*}
By \eqref{eq-compare-of-nomal-derivatives-of-solution-with-theta}, it suffices to show that there exists a constant $\overline{\lambda}>0$ such that
\begin{equation}\label{eq-wanna-of-vp-for-boundary-condition}
-\frac{n-1}{r}+\sqrt{\frac{(n-1)^2}{r^2}-2\lambda}<-\frac{\lambda e^{M_{\lambda,r}}r}{n} \qquad \forall 0<\lambda<\overline{\lambda}.
\end{equation}
Simple computation gives us that
\begin{equation}\label{eq-need-1-of-vp-for-boundary-condition}
\frac{n-1}{r}-\frac{\lambda e^{M_{\lambda,r}}r}{n}>0 \qquad \mbox{if $\lambda<\frac{n(n-1)}{e^{M_{\lambda,r}}r^2}$}
\end{equation}
and, by \eqref{eq-condition-of-special-constant-lambda-1-30},
\begin{equation}\label{eq-need-2-of-vp-for-boundary-condition}
\left(\frac{n-1}{r}-\frac{\lambda e^{M_{\lambda,r}}r}{n}\right)^2>\frac{(n-1)^2}{r^2}-2\lambda \qquad \mbox{if $\lambda<\lambda_1$}.
\end{equation}
Let's define the constant $\overline{\lambda}>0$ by 
\begin{equation*}
\overline{\lambda}=\min\left\{\lambda_1, \frac{(n-1)^2}{2r^2}, \frac{n(n-1)}{e^{M_{\lambda,r}}r^2}\right\}.
\end{equation*}
Then, by \eqref{eq-need-1-of-vp-for-boundary-condition} and \eqref{eq-need-2-of-vp-for-boundary-condition}, \eqref{eq-wanna-of-vp-for-boundary-condition} holds for $0<\lambda<\overline{\lambda}$ and the lemma follows.
\end{proof}
We give next the generalization of the previous lemma for a strictly convex bounded domain $\Omega$ with smooth boundary.
For $x\in\partial\Omega$, we denote by 
\begin{equation*}
r_{\Omega}(x)=\sup\{r\in\R: B_r\subset\Omega \,\, \mbox{and} \,\, B_r\cap\partial\Omega=\{x\}\}\qquad \mbox{and} \qquad r_{\Omega}=\inf_{x\in\partial\Omega}r_{\Omega}(x).
\end{equation*}
We also denote by
\begin{equation*}
R_{\Omega}(x)=\inf\{R\in\R: \Omega\subset B_R \,\, \mbox{and} \,\, B_R\cap\partial\Omega=\{x\}\}\qquad \mbox{and} \qquad R_{\Omega}=\sup_{x\in\partial\Omega}R_{\Omega}(x).
\end{equation*}
Then, we easily conclude that
\begin{equation*}
H(\partial\Omega)\leq \frac{1}{r_{\Omega}} 
\end{equation*}
and 
\begin{equation*}
-\frac{\lambda e^{\overline{M}_{\lambda}} R_{\Omega}}{n}\leq \vp_{\lambda}\leq-\frac{\lambda r_{\Omega}}{n} \qquad \mbox{on $\partial\Omega$}, \qquad \overline{M}_{\lambda}=\left\|\vp_{\lambda}\right\|_{L^{\infty}(\Omega)}
\end{equation*}
for the minimal solution $\vp_{\lambda}$ of the Gelfand's problem \eqref{eq-main-1}. Moreover, by Lemma \ref{lem-estimate-of-mixed-second-derivatives-by-first-normal-direction}, Lemma \ref{lemma-boundary-estimate-for-power-convexity} and Lemma \ref{lem-boundary-condition-on-a-ball-}, one can easily checked that the constant $K$ in the boundary condition \eqref{boundary-condition-of-gelfand-with-PME} is closely related to the difference $R_{\Omega}-r_{\Omega}$ and 
\begin{equation*}
K\to 0 \qquad \mbox{as $\left|R_{\Omega}-r_{\Omega}\right|\to 0$}
\end{equation*}
Hence, following the similar computation as in the proof of Lemma \ref{lem-boundary-condition-on-a-ball-}, we can easily extend the results for the boundary condition on a ball to more general boundary setting.\\
\indent We finish this work with stating the following result.
\begin{thm}
 Let $\Omega$ be a strictly convex bounded subset of $\R^n$ and let $\lambda^{\ast}(\Omega)$ be the extremal value of the Gelfand's problem \eqref{eq-main-1} in $\Omega$. Then, there exists a constant $\overline{\lambda}$ with $0<\overline{\lambda}\leq \lambda^{\ast}(\Omega)$ which is depending on the difference $|R_{\Omega}-r_{\Omega}|$ such that 
 \begin{equation*}
G(\vp_{\lambda},\lambda,\Omega)>0 \qquad \forall 0<\lambda<\overline{\lambda}
\end{equation*}
where $\vp_{\lambda}$ is the minimal solution of the Gelfand's problem in $\Omega$.
\end{thm}

{\bf Acknowledgement} Ki-Ahm Lee was supported by the Korea Research Foundation
Grant funded by the Korean Government(MOEHRD, Basic Research Promotion
Fund)( KRF-2008-314-C00023).

\end{document}